\documentclass{article}

\usepackage{packages}
\usepackage{environments}
\usepackage{commands}
\usepackage{settings}

\newcommand{\F}{\mathfrak{F}}
\newcommand{\G}{\mathfrak{G}}
\newcommand{\M}{\mathfrak{M}}
\newcommand{\N}{\mathfrak{N}}
\newcommand{\Z}{\mathbb{Z}}

\title{On One Application of Asphericity of Presentations}
\author{Maxim Zykin\\
  Faculty of Mechanics and Mathematics, Lomonosov Moscow State University\\
  119991 Moscow, Russia\\
  \texttt{maksim.zykin@math.msu.com}}
\date{}

\begin{document}

\maketitle

\begin{abstract}
I present a direct proof of Lemma~3(a) from O.\,V.~Kulikova’s work on torsion in the group $\F/[\M,\N]$, using only Proposition~1.2 of Chiswell--Collins--Huebschmann on combinatorially aspherical presentations.  In particular, I show that if two presentations satisfy the RC condition and are combinatorially aspherical, then the quotient $\N/[\F,\N]$ is free abelian on the defining relators and central in $\F/[\F,\N]$, whence the extension $\F/[\M,\N]$ is torsion-free.
\end{abstract}

\section{Introduction}
The theory of group presentations with a single defining relation has a long and rich history, dating back to the foundational works of Dehn, Magnus, and others.  One of the central notions in this area is that of an \emph{aspherical presentation}, which ensures that the associated two-dimensional CW--complex is aspherical, i.e.\ has no higher homotopy beyond dimension~1.  Asphericity has powerful algebraic and geometric consequences: it allows one to compute group cohomology, control elements of finite order, and derive structural decompositions of groups.

In a recent paper \cite{Kulikova2024}, O.\,V.~Kulikova studied the torsion in the group
\[
  \F/[\M,\N],
\]
where $\F$ is a free group on a finite set $A$, and $\M,\N$ are the normal closures of two finite families of relators $R_M,R_N\subset \F$.  Her Lemma~3(a) asserts that, if both presentations
\[
  \langle A\mid R_M\rangle
  \quad\text{and}\quad
  \langle A\mid R_N\rangle
\]
satisfy the RC condition (Definition~\ref{def:RC}) and are combinatorially aspherical (Definition~\ref{def:CA}), then the subgroup $\N/[\F,\N]$ is a free abelian central subgroup generated by the images of the defining relators.  Consequently, the extension
\[
  1 \;\longrightarrow\; \N/[\F,\N]\;\longrightarrow\;
  \F/[\M,\N]\;\longrightarrow\;\G\;=\;\F/\N\;\longrightarrow\;1
\]
is torsion-free.

In this note I give a concise, purely algebraic proof of that lemma, bypassing any use of spherical pictures or van Kampen diagrams.  The key input is the relation-module decomposition for combinatorially aspherical presentations from Chiswell--Collins--Huebschmann \cite{CCH1981}.

\section{Basic Definitions}
Let $\F = F(A)$ be a non-abelian free group on a finite alphabet $A$.  If $R\subset \F$ is any (finite) family of cyclically reduced words, write
\[
  \N = \langle\!\langle R\rangle\!\rangle \lhd \F,
  \qquad
  \G = \F/\N.
\]

\begin{definition}[RC condition]\label{def:RC}
A presentation $\langle A\mid R\rangle$ satisfies the \emph{RC} (relator) condition if:
\begin{enumerate}
  \item each $r\in R$ is nonempty and cyclically reduced in $\F$;
  \item no two distinct relators are conjugate to one another or to each other’s inverse in $\F$;
  \item no relator is a proper power in $\F$.
\end{enumerate}
\end{definition}

\begin{definition}[Relation module]
The \emph{relation module} of $\langle A\mid R\rangle$ is the abelianization
\[
  \N/\N' = \N/[\N,\N],
\]
equipped with the natural left $\Z\G$--module structure induced by conjugation in $\F$.
\end{definition}

\begin{definition}[Combinatorial asphericity (CA)]\label{def:CA}
Let $K(A;R)$ be the standard two-dimensional CW--complex of the presentation and let
\[
  C(A;R) = K(A;R)/\sim
\]
be the Cayley complex obtained by identifying those 2--cells corresponding to proper-power relators.  We say $\langle A\mid R\rangle$ is \emph{combinatorially aspherical (CA)} if
\[
  \pi_2\bigl(C(A;R)\bigr)=0.
\]
\end{definition}

\section{Main Lemma}
Suppose $R_M,R_N\subset \F$ both satisfy RC and each of
$\langle A\mid R_M\rangle, \langle A\mid R_N\rangle$ is CA,
with normal closures $\M,\N\lhd\F$.  Then:

\begin{lemma}[Kulikova~\cite{Kulikova2024}]
The subgroup
\[
  \N/[\F,\N]\;\triangleleft\;\F/[\F,\N]
\]
is a free abelian group with basis
$\{r+[\F,\N]\mid r\in R\}$ and is central in $\F/[\F,\N]$.  Consequently, the extension
\[
  1\;\longrightarrow\;\N/[\F,\N]\;\longrightarrow\;\F/[\M,\N]\;\longrightarrow\;\G\;\longrightarrow\;1
\]
is torsion-free.
\end{lemma}

\section{Proof of the Main Lemma}

\begin{proposition}[Chiswell--Collins--Huebschmann {\cite[Prop.~1.2]{CCH1981}}]
Let $\langle A\mid R\rangle$ be concise, RC and CA.  Then
\[
  \N/\N' \;\cong\; \bigoplus_{r\in R} \Z[\G/\langle r\rangle]
\]
as $\Z\G$--modules.
\end{proposition}

\begin{definition}[Augmentation and Coinvariants]
The \emph{augmentation map} $\epsilon:\Z\G\to\Z$ sends $\sum a_g g\mapsto\sum a_g$, with kernel $I_\G=\ker\epsilon$.  For a $\Z\G$--module~$M$, the 
\emph{coinvariants} are
\[
  M_\G = M / I_\G M \cong M\otimes_{\Z\G}\Z.
\]
This quotient identifies $g\cdot m$ with $m$, killing the $G$--action.
\end{definition}

\begin{claim}[Triviality of the Module Action]
The action
\(
  g\cdot(n+[\F,\N]) = (gng^{-1})+[\F,\N]
\)
of $g\in\G$ on $n+[\F,\N]\in\N/[\F,\N]$ is trivial.
\end{claim}
\begin{proof}
By definition $[g,n]=g^{-1}n^{-1}gn \in[\F,\N]$, and normality gives 
$f[g,n]f^{-1}=gng^{-1}n^{-1}\in[\F,\N]$.  Hence in the quotient $gng^{-1}+[\F,\N]=n+ [\F,\N]$.
\end{proof}

\begin{claim}[Centrality]
$\N/[\F,\N]$ is contained in the center of $\F/[\F,\N]$.
\end{claim}
\begin{proof}
For any $n+[\F,\N]$ and $f+[\F,\N]$, 
\[
  (f+[\F,\N])(n+[\F,\N])(f+[\F,\N])^{-1}
  = f\cdot(n+[\F,\N])
  = n+[\F,\N].
\]
Thus they commute.
\end{proof}

\begin{proof}[Proof of Lemma]
From Proposition~1.2,
\[
  \N/\N' \cong \bigoplus_{r\in R}\Z[\G/\langle r\rangle].
\]
Passing to coinvariants:
\[
  \N/[\F,\N]
  = (\N/\N')_\G
  = (\N/\N')/I_\G(\N/\N')
  \cong
  \bigoplus_{r\in R}
  \bigl(\Z[\G/\langle r\rangle]\otimes_{\Z\G}\Z\bigr).
\]
\medskip
\noindent\textbf{Why each summand is $\Z$:}
\begin{itemize}
     
\item $\Z[\G/\langle r\rangle]$ is the free abelian module on the coset set $\G/\langle r\rangle$, with $h\in G$ acting by permuting cosets:
  \[
    h\cdot(g\langle r\rangle)=(hg)\langle r\rangle.
  \]
\item Tensoring with $\Z$ (the trivial module) imposes $h\cdot x=x$ for all $h$, hence identifies every coset basis element with every other.
\item Consequently
  \[
    \Z[\G/\langle r\rangle]\otimes_{\Z\G}\Z 
    \;\cong\;\Z,
  \]
  generated by the common class of any coset.
\end{itemize}
Summing over $r$ gives
\[
  \N/[\F,\N] \;\cong\;\bigoplus_{r\in R}\Z,
\]
freely generated by $\{r+[\F,\N]\}$.  Centrality follows from Claim~2.

\end{proof}

\medskip
\noindent\textbf{Remark.}
Since $\N/[\F,\N]$ is central and torsion-free and $\G=\F/\N$ is torsion-free, any torsion in
$\F/[\M,\N]$ must lie in $\N/[\F,\N]$, hence be trivial.  Thus $\F/[\M,\N]$ is torsion-free.

\section*{Applications}
As in \cite{Kulikova2024}, this lemma shows that HNN--extensions and amalgamated
free products built from CA presentations remain CA and torsion-free.
\section*{Acknowledgments}
I want to thank Mikhail Mikheenko for his useful reviews, and deep knowledge of group theory. I am also grateful to Anton~A.~Klyachko for being my supervisor for almost two years.

\end{document}